\documentclass[12pt,reqno]{amsart}
\usepackage{comment,tikz,mathrsfs,amssymb}  

\def\m{\mathfrak{m}}

\usepackage{multicol}
\usepackage{xcolor}
\usepackage{hyperref}

\newcommand{\vertleq}{\rotatebox{270}{$\!\!\!\leq\,\,$}}

\numberwithin{equation}{section}

\newtheorem{theorem}{Theorem}[section]

\newtheorem{lemma}[theorem]{Lemma}

\newtheorem{proposition}[theorem]{Proposition}

\theoremstyle{definition}
\newtheorem{example}[theorem]{Example}

\newtheorem{remark}[theorem]{Remark}

\newtheorem{question}[theorem]{Question}

\newcommand{\card}{\mathrm {card\, }}
\newcommand{\st}{\mathrm {{\bf st}\, }}
\newcommand{\rt}{\mathrm {{\bf rt}\, }}
\newcommand{\sar}{\mathrm {{\bf s}\, }}
\newcommand{\war}{\mathrm {{\bf w}\, }}
\newcommand{\mar}{\mathrm {{\bf m}\, }}
\newcommand{\SD}{\mathrm {SD}\,}
\newcommand{\sym}[1]{\mbox{$\mathbf{S}(#1)$}}
\newcommand{\rees}[1]{\mbox{$\mathbf{R}(#1)$}}
\newcommand{\reesw}[2]{\mbox{$\mathbf{R}(#1;#2)$}}

\newcommand{\agrw}[2]{\mbox{$\mathbf{G}(#1;#2)$}}

\def\height{\hbox{\rm ht}\,}
\def\grade{\hbox{\rm grade}\,}
\def\rad{\hbox{\rm rad}\,}

\begin{document}
\title[Sifted degrees and Artin-Rees numbers]{Sifted degrees of the equations of the Rees module
and their connection with the Artin-Rees numbers. 
}
 
\author{Philippe Gimenez}
 \address{IMUVA-Mathematics Research Institute, Universidad de Valladolid, 47011 Valladolid, Spain.}
 \email{pgimenez@uva.es}
 
\author{Francesc Planas-Vilanova}
 \address{Departament de Matemàtiques, Universitat Politècnica de Catalunya, Diagonal
647, ETSEIB, E-08028 Barcelona, Catalunya}
 \email{francesc.planas@upc.edu}

\thanks{The first author was partially supported by the grant PID2022-137283NB-C22 funded by Spanish MICIU/AEI/10.13039/501100011033 and by ERDF/EU. 
\\The second author was partially supported by projects reference PID2019-103849GB-I00 and PID2023-146936NB-I00 funded
by the Spanish State Agency MICIU/AEI/10.13039/501100011033. He is also supported by an AGAUR
grant 2021 SGR 00603. 
}

\begin{abstract}
Let $A$ be a noetherian ring, $I$ an ideal of $A$ and $N\subset M$ finitely generated $A$-modules. The relation type of $I$ with respect to $M$, denoted by $\rt(I;M)$, is the maximal degree in a minimal generating set of relations of the Rees module $\reesw{I}{M}=\oplus_{n\geq 0}I^nM$. It is a well-known invariant that gives a first measure of the complexity of $\reesw{I}{M}$. 
To help to measure this complexity, we introduce the sifted type of $\reesw{I}{M}$, denoted by $\st(I;M)$, 
a new invariant which counts the non-zero degrees appearing in a minimal generating set of relations of $\reesw{I}{M}$. 
Just as the relation type $\rt(I;M/N)$ is closely related to the strong Artin-Rees number $\sar(N,M;I)$, it turns out that the sifted type $\st(I;M/N)$ is closely related to the medium Artin-Rees number $\mar(N,M;I)$, a new invariant which lies in between the weak and strong Artin-Rees numbers of $(N,M;I)$.
We illustrate the meaning, interest and mutual connection of $\mar(N,M;I)$ and $\st(I;M)$ with some examples. 
\end{abstract}

\maketitle

\section*{Preamble}

The Rees algebra and the other so-called blowing-up algebras have been present in some way, centrally or collaterally, in many of Wolmer V. Vasconcelos' publications, both articles and books,
since the early 1980s, when he began his long and very fruitful collaboration with Aron Simis.
They are the authors with Jürgen Herzog of the foundational papers \cite{hsv1, hsv2, hsv3} and, together with other relevant and influential algebrists,  
they located the study of these algebras as a central topic of commutative algebra.
The two authors of the present note have learned about Rees algebras (and many other topics) mainly by reading Wolmer's contributions and from many stimulating conversations with him over three decades after they had the opportunity to meet him (and each other) in Trieste in 1992. We would like to thank the editors of this volume, and especially Aron Simis, for being so patient in waiting for this contribution and for giving us the opportunity, by paying tribute to Wolmer, to work together after more than 30 years of friendship. 

\vskip 8pt
This paper is dedicated to the memory of Wolmer V. Vasconcelos.

\section{Introduction}

Given a noetherian ring $A$ and an ideal $I$ in $A$, the Rees algebra of $I$ is the subring $\rees{I}=\oplus_{n\geq 0}I^nt^n$ of the polynomial ring $A[t]$. It provides an algebraic realization of the notion of blowing-up a variety along a subvariety and hence plays an important role in the theory of singularities in algebraic geometry. From an algebraic point of view, it is an ubiquitous object linked to many constructions and problems in commutative algebra. 
It is related to the symmetric algebra $\sym{I}$ of $I$ through the graded canonical surjection $\alpha:\sym{I}\to\rees{I}$.
While the symmetric algebra takes care of the linear relations among the generators of the ideal $I$, i.e., its first syzygies, the Rees algebra takes care of all their homogeneous relations, i.e., as Vasconcelos claims in \cite[p. 195]{wolmer2}, the first syzygies of all the powers of $I$. 
More precisely, if the ideal $I$ of $A$ is generated by $x_1,\ldots,x_n$, one can consider the surjection $\varphi$ from the ring of polynomials $A[X_1,\ldots,X_n]$ onto $\rees{I}$ defined by $X_i\mapsto x_it$. The elements in $\ker(\varphi)$ are called the equations of the Rees algebra $\rees{I}$. The maximal degree in a minimal generating set of $\ker(\varphi)$ is called the relation type of $I$ and denoted by $\rt(I)$. It does not depend on the choosen generating set of $I$; see, e.g., \cite{fp-camb,raghavan,wolmer2,wang1}.
The description of $\ker(\varphi)$ is, in general, a difficult problem. In this sense, the relation type gives a first measure of the complexity of the problem of describing  the equations of the Rees algebra.
To complement the information given by the relation type, we introduce a second measure which is the number of non-zero degrees in a minimal generating set of $\ker(\varphi)$.
We call this new invariant the {\em sifted type of $I$}, and we denote it by $\st(I)$. One would like to know $\rt(I)$ and $\st(I)$ before determining the equations of the Rees algebra of $I$.

From this point of view, the simplest ideals are those whose relation type (and sifted type) is equal to 1. Note that an ideal $I$ has relation type 1 if and only if the canonical morphism $\alpha:\sym{I}\to\rees{I}$ is an isomorphism.
These ideals are said to be of linear type. According to \cite[p. 30]{wolmer2}, the terminology was introduced by Robbiano and Valla. Ideals of linear type have been extensively studied in the 1980s \cite{costa1,costa2,hmv,hrz,hsv1,hsv2,hsv3,huneke1,huneke2,kuhl,trung1,valla}, but the 
first important family of ideals of linear type is already given by Micali in the early 1960s \cite{micali1,micali2}: if $I$ is 
generated by a regular sequence, then it is of linear type. 
In \cite{huneke1} Huneke introduces the notion of $d$-sequence, a weaker condition to that of regular sequence, and shows that an ideal generated by a $d$-sequence is of linear type, a result proved  simultaneously by Valla in \cite{valla}. 
Other classes of ideals with low relation type have been studied. For example, the notion of quadratic sequence is introduced in \cite{raghavan} where ideals generated by such a sequence are shown to have relation type 2.
On the other hand, an ideal $I$ is said to be syzygetic if $\alpha_2:{\bf S}_2(I)\to I^2$ is an isomorphism. This is a strictly weaker condition to that of being of linear type.
Syzygetic ideals were introduced in \cite{mr} and have been studied or used, e.g., in \cite{br,herzog, hsv1, hsv2, sv, santi}. Describing the syzygetic condition we have, e.g., the following two results: 
if $R$ is local of depth 2 and
$I$ is an ideal of heigth 2 and projective dimension one, then $I$ is syzygetic if and only if $I$ is generated by a regular sequence of two elements 
\cite[Proposition~2.7]{hsv2}; if $R$ is regular local, with $1/2\in R$, then a Gorenstein ideal of height 3 is syzygetic \cite[Proposition~2.8]{hsv2}. (See \cite[Lemma~2.1]{fp-2022} for a recent generalization of this result.) 

There are several powerful and efficient tools to obtain some equations of the Rees algebra, though very rarely they provide the whole generating set of its equations.
The so-called Jacobian dual is one of these tools. 
It already appears in \cite{wolmer1} (see also \cite[Chapter 8]{wolmer2}) 
for ideals and has been further developed for modules in \cite{suv95}. It works as follows:
if $J$ is the $n\times r$ matrix of the first syzygies of $I=\langle x_1,\ldots,x_n\rangle$ and if $\varphi:A[X_1,\ldots,X_n]\rightarrow\rees{I}$ is a polynomial presentation of $\rees{I}$ with $L:=\ker(\varphi)$,
then $L_1$ is generated by the entries of the matrix 
$\begin{pmatrix}X_1&\cdots&X_n\end{pmatrix}\cdot J$. 
If $a=\langle a_1,\ldots,a_d\rangle$ is an ideal of $A$ containing all the entries of $J$, one can also write these generators of $L_1$ as the entries of the matrix 
$\begin{pmatrix}a_1&\cdots&a_d\end{pmatrix}\cdot B(J)$, for some
non-unique $d\times r$ matrix $B(J)$. In other words, one can think of $J$ and $B(J)$ as the Jacobian matrices of the ideal $L_1$ with respect to the
``variables" $X_1,\ldots,X_n$ and $a_1,\ldots,a_d$ respectively. Using Cramer's rule, it is easy to show that the Fitting ideal $I_r(B(J))$ is contained in $L$. Thus, one gets equations of the Rees algebra using this simple technique. When $L=\langle L_1,I_r(B(J))\rangle$, the ideal $I$ is said to have the expected equations \cite[Definition~1.8]{wolmer5}.
This has been studied for ideals of small codimension in \cite{mu}. See also \cite{gsvv} for a successful application of this technique to describe the entire Rees algebra.

As mentioned before, the relation type of an ideal $I$, combined with its sifted type, is a measure of the complexity in describing the equations of its Rees algebra and, hence, it can also be 
seen as a measure of the complexity of the ideal $I$ itself. 
This also occurs with the Castelnuovo-Mumford regularity of $I$, 
which measures the difficulty in constructing the syzygies (of any order) of $I$,
and hence guesses the degree of complexity of $I$.
For the connection between the relation type and the Castelnuovo-Mumford regularity, as well as their relationship with other invariants, see, e.g.,
\cite{bg, gp, jk, jm, mp, fp-camb, schenzel, trung2, trung3,wolmer4}. 
It is also worth mentioning that the equations of the Rees algebra are involved in some of the examples constructed in \cite{mcp} to refute the challenging Eisenbud-Goto regularity conjecture. 

The Rees algebra $\rees{I}$ is also involved in the simple proof of the Artin-Rees Lemma for ideals given in \cite[Theorem. A.1.6]{wolmer3}. Recall that the general version of the Artin-Rees Lemma for modules states that, 
for any ideal $I$ in $A$ and any two finitely generated $A$-modules $N\subseteq M$, there exists an integer $s\geq 1$ such that, 
$I^nM\cap N=I^{n-s}(I^sM\cap N)$, for all $n\geq s$. 
The notion of Rees module $\reesw{I}{M}=\oplus_{n\geq 0}I^nM$ of an ideal $I$ with respect to a module $M$
can be used in the proof of this result (see, e.g., \cite[Lemma~1.1]{fp-lisboa}).
There is a well-known connection between the Artin-Rees Lemma and the relation type, which roughly speaking establishes that 
the aforementioned integer $s$ linked to the triple $(N,M;I)$ is bounded above by the relation type of the Rees module $\reesw{I}{M/N}$ (see, e.g., 
\cite{hsv2, huneke1, huneke2, lai, fp-lisboa, wang1} and, in particular, \cite[Theorem~2]{fp-crelle}).
We would like to mention that in \cite{hsv2}, the authors observe how being of linear type gives the Artin-Rees Lemma `on the nose' as already observed in \cite{huneke2} for ideals generated by a $d$-sequence.

The Artin-Rees Lemma has raised many relevant questions of a uniform nature.
The pair of finitely generated $A$-modules $(N,M)$ with $N\subseteq M$ is said to have the strong uniform Artin–Rees property if there exists an integer $s\geq 1$ (depending on $M$ and $N$), such that for all
$n\geq s$ and all ideals $I$, $I^nM\cap N=I^{n-s}(I^sM\cap N)$.
In \cite{eh}, Eisenbud and Hochster ask if, when $A$ is an affine ring, any pair $(N,M)$ has the strong uniform Artin–Rees property with respect to maximal ideals (i.e., restricting the property to all maximal ideals of $A$).
This question has been answered by Duncan and O'Carroll in \cite{do} when $A$ is an excellent ring. Previously, O'Carroll has shown in \cite{ocarroll1} that if $A$ is an excellent ring, any pair $(N,M)$ has the weak uniform Artin–Rees property with respect to maximal ideals, where the weak uniform property asks for the weakest condition $I^nM\cap N\subseteq I^{n-s}N$. 
The same author showed later on in \cite{ocarroll2} (see also \cite{ocarroll3}) that over an arbitrary noetherian ring $A$, any pair $(N,M)$ has the strong Artin-Rees uniform property with respect to principal ideals: there exists an integer $s\geq 1$ such that for all
$n\geq s$ and all $x\in A$, $x^nM\cap N=x^{n-s}(x^sM\cap N)$. 
Inspired by these results, Huneke proves in \cite{huneke3} several kinds of uniform behaviour and conjectures
that excellent noetherian rings of finite Krull dimension have the weak uniform Artin-Rees property (see also \cite{huneke4} and \cite{trivedi}). For the uniform Artin-Rees property with respect to prime ideals, see \cite{op}. With respect to the whole set of ideals of $A$, see \cite{fp-crelle} (and also \cite{striuli}).
Going back to the Rees algebra, a uniform property of the relation type is raised by Huneke: if $A$ is a complete local equidimensional noetherian ring,
does there exist a uniform number $N$ such that for any ideal $I$ generated by a system of parameters, 
$\rt(I)\leq N$? Partial positives answers to this question are given in \cite{agh,lai,wang1,wang2}. 

In this note, we focus on the relationship between 
the Artin-Rees numbers and the degrees of the defining equations of the Rees module of an ideal with respect to a module. Associated to a triple $(N,M;I)$ formed by two finitely generated modules $N\subseteq M$ and an ideal $I$, there are two known invariants, the strong and the weak Artin-Rees numbers.
In Section \ref{sec-invariants}, we recall the definitions of these two invariants and introduce the {\em medium Artin-Rees number}, a new invariant located 
between the weak and the strong Artin-Rees numbers (Proposition \ref{prop-arnumbers}).
On the other hand, we recall the definition of the relation type of an ideal $I$ with respect to a module $P$, 
and introduce the {\em sifted type of $I$ with respect to $P$}, a smaller invariant. 
The connection among these five invariants is explained in Theorem~\ref{thm-summ}.
In Section \ref{sec-firstEx}, we give several families of examples to illustrate the behavior of these invariants. In Section \ref{sec-fibo}, we focus on the case of a polynomial ideal $I$ in 2 variables minimally generated by 3 monomials of the same degree with no common factor. 
We take advantage of the minimal bigraded free resolution of $\rees{I}$ constructed in \cite{cda}, to compute the values of the relation type and the sifted type of $I$ (Proposition~\ref{3monomials}). 
(For a very recent generalization of \cite{cda}, see \cite{ioss}, where the three monomials can have different degrees).
Proposition~\ref{3monomials} allows us to give several interesting examples, 
where the difference between the relation type and the sifted type can be arbitrarily large.
In particular, as a consequence of Example \ref{fibonacci}, we deduce that, given any integer $\ell$, we can easily obtain an ideal of sifted degree equal to $\ell$. Finally, in Section \ref{sec-moreEx} we revisit three classical examples to stand out their sifted type and relation type.

We are aware that the list of references included here is by no means exhaustive, it is just a brief selection limited by our knowledge and time frame.

\section{Some invariants linked to the triple 
\texorpdfstring{$(N,M;I)$}{(N,M;I)}.}\label{sec-invariants}

Let $A$ be a noetherian ring, $I$ an ideal of $A$, and $N\subseteq M$ finitely generated $A$-modules. 
Let us recall (as well as introduce) the definition of some (new) 
invariants associated to these modules and to the ideal $I$ 
(see, e.g., \cite{hsv2}, \cite{huneke1}, \cite{fp-lisboa} and the references there-in).

\vspace*{0.3cm}

For ease of further reference, we denote and call the $n$-th Artin-Rees module, $n\geq 1$, the quotient:
\begin{eqnarray*}
E(N,M;I)_n\stackrel{\rm def}{=}(I^nM\cap N)/I(I^{n-1}M\cap N).
\end{eqnarray*}

\vspace*{0.3cm}

Associated to the triple $(N,M;I)$ we have the following three invariants.  

\vspace*{0.3cm}

\noindent $\bullet$ \underline{The strong Artin-Rees number $\sar(N,M;I)$}. 

\vspace*{0.3cm}

\noindent The Artin-Rees Lemma states that there exists an integer $s\geq 1$ such that
\begin{eqnarray}\label{eq-ar}
I^nM\cap N=I^{n-s}(I^sM\cap N),\mbox{ for all }n\geq s.
\end{eqnarray}
In particular, for all $n\geq t\geq s$, 
\begin{eqnarray*}
I^nM\cap N=I^{n-s}(I^sM\cap N)=I^{n-t}(I^{t-s}(I^sM\cap N))=I^{n-t}(I^tM\cap N).
\end{eqnarray*}
Thus, if the equality \eqref{eq-ar} holds for $s\geq 1$, then it also holds for any $t\geq s$.
Hence, it is natural to define the strong Artin-Rees number of $(N,M;I)$ as:
\begin{eqnarray*}
\sar(N,M;I)&\stackrel{\rm def}{=}&\min\{s\geq 1\mid I^nM\cap N=I^{n-s}(I^sM\cap N),\mbox{ for all }n\geq s\}\\
&=&\min\{s\geq 1\mid I^nM\cap N=I(I^{n-1}M\cap N),\mbox{ for all }n\geq s+1\}\\
&=&\min\{s\geq 1\mid E(N,M;I)_n=0,\mbox{ for all }n\geq s+1\}.
\end{eqnarray*}

\vspace*{0.3cm}

\noindent $\bullet$ \underline{The weak Artin-Rees number $\war(N,M;I)$}.

\vspace*{0.3cm}

\noindent As a consequence of the Artin-Rees Lemma, one deduces that there exists an integer $s\geq 1$ such that
$I^nM\cap N\subseteq I^{n-s}N$ for all $n\geq s$. In particular, for all 
$n\geq t\geq s$, 
\begin{eqnarray*}
I^nM\cap N\subseteq I^{n-s}N\subseteq I^{n-t}N.
\end{eqnarray*}
Thus, $I^nM\cap N\subseteq I^{n-t}N$, for all $n\geq t$.
The weak Artin-Rees number is then defined as:
\begin{eqnarray*}
\war(N,M;I)&\stackrel{\rm def}{=}&\min\{s\geq 1\mid I^nM\cap N\subseteq I^{n-s}N\mbox{ for all }n\geq s\}.
\end{eqnarray*}

\vspace*{0.3cm}

\noindent $\bullet$ \underline{The medium Artin-Rees number $\m(N,M;I)$}.

\vspace*{0.3cm}

\noindent In between the strong and the weak Artin-Rees numbers, we consider a new invariant, the {\em medium Artin-Rees number}, defined as follows: 
\begin{eqnarray*}
\mar(N,M;I)&\stackrel{\rm def}{=}&1+\card\{n\geq 2\mid E(N,M;I)_n\neq 0\}.
\end{eqnarray*}
\begin{proposition}\label{prop-arnumbers}
\begin{eqnarray*}
\war(N,M;I)\leq \mar(N,M;I)\leq \sar(N,M;I).
\end{eqnarray*}
\end{proposition}
\begin{proof}
Note that, if $\sar(N,M;I)=r$, $r\geq 1$, then $E(N,M;I)_n=0$, for all $n\geq r+1$. Thus,
the set of indices $n\geq 2$ with $E(N,M;I)_n\neq 0$ is included in $\{2,\ldots ,r\}$. Therefore, 
$\mar(N,M;I)\leq r=\sar(N,M;I)$. 

Now, let us prove that $\war(N,M;I)\leq \mar(N,M;I)$.
Observe that, if for a given pair of integers $n>m\geq 1$,
$E(N,M;I)_{m+1}=\ldots=E(N,M;I)_{n}=0$, then 
\begin{eqnarray}\label{eq-delta}
I^{n}M\cap N=I(I^{n-1}M\cap N)=\ldots =I^{n-m}(I^mM\cap N)\subseteq I^{n-m}(I^{m-1}M\cap N).
\end{eqnarray}
If $\mar(N,M;I)=1$, then $E(N,M;I)_n=0$, for all $n\geq 2$. 
Using \eqref{eq-delta}, with $m=1$, we get $I^nM\cap N\subseteq I^{n-1}N$, for all $n\geq 2$,
so $\war(N,M;I)\leq 1=\mar(N,M;I)$. 

Suppose that $\mar(N,M;I)=r\geq 2$. Then, 
there exist $2\leq m_2<\ldots<m_r$ with $E(N,M;I)_{m_i}\neq 0$, for $i=2,\ldots, r$, 
and such that $E(N,M;I)_n=0$, for all $n\geq 2$, $n\not\in\{m_2,\ldots,m_r\}$. 
Let $n\geq r$. If $2\leq n<m_2$, then, for all $2\leq m\leq n$, 
$E(N,M;I)_m=0$, so $I^nM\cap N=I^{n-1}(IM\cap N)$.
In particular, $I^nM\cap N\subseteq I^{n-1}N\subseteq I^{n-r}N$. If $m_2\leq n<m_3$, 
using twice \eqref{eq-delta}, we get:
\begin{multline*}
I^nM\cap N=I^{n-m_2}(I^{m_2}M\cap N)\subseteq I^{n-m_2}(I^{m_2-1}M\cap N)
\\=I^{n-m_2}(I^{m_2-2}(IM\cap N))=I^{n-2}(IM\cap N).
\end{multline*}
In particular, $I^nM\cap N\subseteq I^{n-2}N\subseteq I^{n-r}N$. Suppose that $m_{r-1}\leq n<m_r$. 
As before, using \eqref{eq-delta}, and by induction,
\begin{multline}\label{eq-penultima}
I^nM\cap N=I^{n-m_{r-1}}(I^{m_{r-1}}M\cap N)\subseteq I^{n-m_{r-1}}(I^{m_{r-1}-1}M\cap N)
\\\subseteq I^{n-m_{r-1}}(I^{m_{r-1}-(r-1)}(IM\cap N))=I^{n-(r-1)}(IM\cap N).
\end{multline}
In particular, $I^nM\cap N\subseteq I^{n-(r-1)}N\subseteq I^{n-r}N$. Finally, suppose that $n\geq m_r$. 
Again, using \eqref{eq-delta} and \eqref{eq-penultima}, we deduce
\begin{multline*}
I^nM\cap N=I^{n-m_{r}}(I^{m_r}M\cap N)\subseteq I^{n-m_r}(I^{m_r-1}M\cap N)\\
\subseteq I^{n-m_r}(I^{(m_r-1)-(r-1)}(IM\cap N))=I^{n-r}(IM\cap N).
\end{multline*}
Therefore, $I^nM\cap N\subseteq I^{n-r}N$, for all $n\geq r$, so $\war(N,M;I)\leq \mar(N,M;I)$.
\end{proof}

\vspace*{0.4cm}

Now we turn to relate these Artin-Rees numbers with the degrees of the equations of the Rees modules. 
Let $\sym{I}$ be the symmetric algebra of $I$. Given a finitely generated $A$-module $P$, 
consider the Rees module $\reesw{I}{P}=\oplus_{n\geq 0}I^nP$ of $I$ with respect to $P$. Let  $\alpha:\sym{I}\to\rees{I}$ be the canonical graded surjective morphim from the symmetric algebra of $I$ 
to the Rees algebra of $I$.
Let $\pi:\rees{I}\otimes P\rightarrow\reesw{I}{P}$ be the natural graded surjective morphim of graded $\sym{I}$-modules, given by the surjection $I^{n}\otimes P\to I^nP$. Then,
\begin{eqnarray*}
\gamma:\sym{I}\otimes P\stackrel{\alpha\otimes 1}{\longrightarrow}\rees{I}\otimes P
\stackrel{\pi}{\rightarrow}\reesw{I}{P}
\end{eqnarray*}
is a surjective graded morphism of graded $\sym{I}$-modules. 
The module of effective $n$-relations of $I$ with respect to $P$, $n\geq 2$, is defined as 
\begin{eqnarray*}
E(I;P)_n\stackrel{\rm def}{=}\ker(\gamma_n)/{\bf S}_1(I)\cdot \ker(\gamma_{n-1}),
\end{eqnarray*}
where $\gamma_n:{\bf S}_n(I)\otimes P\to I^nP$ stands for the $n$-th graded component of $\gamma$.

If  $\rho:\sym{F}\otimes G\to\reesw{I}{P}$ is any other symmetric presentation (e.g., $F$ and $G$ free modules),
then one can show that
$\ker(\rho_n)/{\bf S}_1(F)\cdot \ker(\rho_{n-1})$ is isomorphic to $E(I;P)_n$ 
(see, e.g., \cite[Lemma~2.3 and Definition~2.4]{fp-crelle}). 
Thus, $E(I;P)_n$ describes all those relations of degree $n$ of $\reesw{I}{P}$ which can not be obtained 
by relations of lower degree. If $P=A$, we just write $E(I)$ for $E(I;A)$.

The next result links effective relations with Artin-Rees modules. 

\begin{proposition}\label{prop-key}
For each $n\geq 2$, we have the short exact sequence of $A$-modules:
\begin{eqnarray}\label{eqkey}
E(I;M)_{n}\rightarrow E(I;M/N)_{n}\stackrel{\varphi}{\rightarrow}
E(N,M;I)_{n}\rightarrow 0.
\end{eqnarray}
If $I$ is generated by the elements $x_1,\ldots,x_r$, then the morphism $\varphi$ sends the class of an effective $n$-relation 
$\sum_{|\alpha|=n}\overline{m}_\alpha T^{\alpha}\in M/N[T_1\ldots,T_r]$
to the class of the element $\sum_{|\alpha|=n}m_{\alpha}x^{\alpha}$ in the module $E(N,M;I)_n=(I^{n}M\cap N)/I(I^{n-1}M\cap N)$.
\end{proposition}
\begin{proof}
The existence of the exact sequence follows from \cite[Lemma~2.3 and Theorem~2]{fp-crelle} or \cite[Section~6]{fp-lisboa}. 
To calculate the image 
of an element of $E(I;M/N)_n$ in $E(N,M;I)_n$, just trace back the definition of the morphism $\varphi$ 
in \cite[Lemma~2.3]{fp-crelle}.
\end{proof}

In connection with the degrees of the effective relations, one has two invariants. 
The first one is classical, while the second is new.

\vspace*{0.3cm}

\noindent $\bullet$ \underline{The relation type $\rt(I;P)$ of $I$ with respect to $P$}.

\vspace*{0.3cm}

\noindent The relation type of $I$ with respect to $P$ can be defined as
\begin{eqnarray*}
\rt(I;P)&\stackrel{\rm def}{=}&\min\{\, r\geq 1\mid E(I;P)_{n}=0, \mbox{ for all } n\geq r+1\}.
\end{eqnarray*}
So, $\rt(I;P)$ is the maximum degree in a minimal generating set of relations of the Rees module 
$\reesw{I}{P}$. If $P=A$, we just write $\rt(I)$ for $\rt(I;A)$.
From the exact sequence \eqref{eqkey}, one easily deduce the inequalities
(see \cite[Theorem~2]{fp-crelle}):
\begin{eqnarray}\label{arrt}
\sar(N,M;I)\leq \rt(I;M/N)\leq {\rm max}(\rt(I;M),\sar(N,M;I)).
\end{eqnarray}

\vspace*{0.3cm}

\noindent $\bullet$ \underline{The sifted type $\st(I;P)$ of $I$ with respect to $P$}.

\vspace*{0.3cm}

\noindent The set $\SD(I;P)$ of sifted degrees of $I$ with respect to $P$ is defined as the set of integers 
\begin{eqnarray*}
\SD(I;P)&\stackrel{\rm def}{=}&\{1\}\cup \{n\geq 2\mid E(I;P)_n\neq 0\}.
\end{eqnarray*}
The {\em sifted type of $I$ with respect to $P$} is defined as
\begin{eqnarray*}
\st(I;P)&\stackrel{\rm def}{=}&\card(\SD(I;P))=1+\card \{\, n\geq 2\mid E(I;P)_{n}\neq 0\}.
\end{eqnarray*}
We add the integer 1 in $\SD(I;P)$
to count the possible non-zero syzygies of $IP$, which can be seen as the relations of degree 1. 
Even if there are no relations in degree 1, we want $\st(I;P)$, as well as $\rt(I;P)$, to be at least 1.
As before, if $P=A$, then we just write $\SD(I)$ and $\st(I)$. 

From the definition and the exact sequence \eqref{eqkey}, it is readily seen that:

\begin{lemma}
\begin{eqnarray*}
\st(I;P)\leq \rt(I;P).
\end{eqnarray*}   
\end{lemma}
\begin{proof}
Suppose that $\rt(I;P)=r$. Then $E(I;P)_n=0$, for all $n\geq r+1$. Hence,
$\SD(I;P)\subseteq\{1,\ldots,r\}$, $\st(I;P)\leq r$ and $\st(I;P)\leq \rt(I;P)$. 
\end{proof}

In an analogy to \cite[Theorem~2]{fp-crelle}, recalled previously in \eqref{arrt}, we have the following result. 

\begin{proposition}\label{stmenor}
\begin{eqnarray*}\label{arnt}
\mar(N,M;I)\leq \st(I;M/N)\leq \st(I;M)+\mar(N,M;I))-1.
\end{eqnarray*}   
\end{proposition}
\begin{proof}
Let $\st(I;M/N)=r$, where $\SD(I;M/N)=\{1,m_2,\ldots,m_r\}$, $1<m_2<\ldots<m_r$. 
Using the exact sequence \eqref{eqkey}, we deduce that  
$E(N,M;I)_n=0$, for all $n\geq 2$, $n\not\in\{m_2,\ldots,m_r\}$. Therefore,
$\mar(N,M;I)\leq 1+(r-1)=r=\st(I;M/N)$.

Now, let $\st(I;M)=s$, so $\SD(I;M)=\{1,n_2,\ldots,n_s\}$, 
for some $1<n_2<\ldots<n_s$. 
Let $\mar(N,M;I)=t$. Thus, there are $2\leq m_2<\ldots<m_t$ such that 
$E(N,M;I)_{m_{j}}\neq 0$, for $j=2,\ldots,t$, and $E(N,M;I)_m=0$, for all $m\geq 2$, $m\not\in \{m_2,\ldots,m_t\}$.
Using the exact sequence \eqref{eqkey}, we deduce that
$E(I;M/N)_n=0$, for all $n\geq 2$, $n\not\in\{n_2,\ldots,n_s\}\cup\{m_2,\ldots,m_t\}$. Therefore,
$\st(I;M/N)\leq 1+(s-1)+(t-1)=\st(I;M)+\mar(N,M;I)-1$.
\end{proof}

A special case takes place when $I$ is an ideal of linear type with respect to $M$, 
i.e., when $\rt(I;M)=1$ (see, e.g., \cite{hsv2}). 

\begin{proposition}\label{prop-linear}
Let $I$ be an ideal of linear type with respect to $M$. Then $E(I;M/N)_n=E(N,M;I)_n$, for all $n\geq 2$. 
In particular, $\sar(N,M;I)=\rt(I;M/N)$ and $\mar(N,M;I)=\st(I;M/N)$. 
\end{proposition}
\begin{proof}
Indeed, by hypothesis, $E(I;M)_n=0$, for all $n\geq 2$. Using the exact sequence \eqref{eqkey}, we deduce that
$E(I;M/N)_n=E(N,M;I)_n$, for all $n\geq 2$. Hence, 
\begin{eqnarray*}
\rt(I;M/N)&=&\min\{\, r\geq 1\mid E(I;M/N)_{n}=0, \mbox{ for all } n\geq r+1\}\\
&=&\min\{\, r\geq 1\mid E(N,M;I)_{n}=0, \mbox{ for all } n\geq r+1\}=\sar(N,M;I)\mbox{ and}\\
\st(I;M/N)&=&1+\card\{n\geq 2\mid E(I;M/N)_n\neq 0\}\\&=&1+\card\{n\geq 2\mid E(N,M;I)_n\neq 0\}=\mar(N,M;I). 
\end{eqnarray*}
\end{proof}

Summarizing the previous results, we have the following:

\begin{theorem}\label{thm-summ}
If $A$ is a noetherian ring, for any ideal $I$ in $A$ and any pair of finitely generated $A$-modules $N\subseteq M$, the following diagram of inequalities holds:
\begin{eqnarray*}
\begin{matrix}
\war(N,M;I) & \leq & \mar(N,M;I) &\leq & \sar(N,M;I)\\ 
&& \vertleq && \vertleq \\
&& \st(I;M/N) & \leq & \rt(I;M/N).
\end{matrix}
\end{eqnarray*}
\end{theorem}

Before finishing the section we emphasize the following fact.

\begin{remark}
The notion of a module of effective relations, as well as its linked concepts of relation type and sifted type, 
can readily be extended to any standard module (see, e.g., \cite{gp} or \cite{fp-crelle}). 
It is of great interest to consider the associated graded module
$\agrw{I}{M}=\oplus_{n\geq 0}I^nM/I^{n+1}M$
of an ideal $I$ of $A$ with respect to an $A$-module $M$ and its module of effective $n$-relations $E(\agrw{I}{M})_n$. 
It is shown that, for every $n\geq 2$, there is an exact sequence of $A$-modules (see \cite[Proposition~3.3]{fp-camb}
or \cite[Proposition~3.2]{gp}):
\begin{eqnarray}\label{eq-agr}
E(I;M)_{n+1}\stackrel{\psi_n}{\longrightarrow}E(I;M)_n\rightarrow E(\agrw{I}{M})_n\to 0.
\end{eqnarray}
If $I$ is generated by $x_1,\ldots,x_r$, then the morphism $\psi_n$ sends the class of an effective relation 
$\sum_{|\alpha|=n+1}m_\alpha T^{\alpha}$ in $M[T_1\ldots,T_r]$
to the class of the element
$\sum_{|\alpha|=n+1}x_im_{\alpha}T^{\alpha_i}$ in the module $E(M;I)_n$, where $T_i$ divides $T^{\alpha}$ and 
$T^{\alpha_i}:=T^{\alpha}/T_i$ (see the downgrading morphism in \cite{hsv2}).  

From the exact sequence \eqref{eq-agr}, one deduces that $\rt(\agrw{I}{M})=\rt(I;M)$ 
(see \cite[Proposition~3.3]{fp-camb} and \cite[Remark~2.7]{fp-crelle}). 
Hence, it is natural to ask 
whether $\st(\agrw{I}{M})$ is equal to $\st(I;M)$. 
We see in Example~\ref{ex-principal} that the answer is negative in general. 
\end{remark}

\section{First Examples}\label{sec-firstEx}

In this section we give some simple examples. We begin by showing that $\st(\agrw{I}{M})$ 
might be strictly smaller than $\st(I;M)$.

\begin{example}\label{ex-principal}
Let $I=(a)$ be a principal ideal of a noetherian local ring $A$ and let $M$ be a finitely generated $A$-module.
Then, $E(I;M)_n=(0:a^n)/(0:a^{n-1})$, where the colon module is defined as $(0:a^n):=\{x\in M\mid a^nx=0\}$. 
In particular, if $E(I;M)_n=0$, then $E(I;M)_m=0$, for all $m\geq n$. Hence, $\st(I;M)=\rt(I;M)$. 

Note that the morphism $\psi_n:E(I;M)_{n+1}\to E(I;M)_n$ defined in \eqref{eq-agr}
sends the class of an element $x\in (0:a^{n+1})$ to the class of $ax$ in $E(I;M)_n=(0:a^n)/(0:a^{n-1})$. Thus, 
the image of $E(I;M)_{n+1}$ through $\psi_n$ is
\begin{eqnarray*}
\psi_n(E(I;M)_{n+1})=\frac{a(0:a^{n+1})+(0:a^{n-1})}{(0:a^{n-1})}\subseteq \frac{(0:a^n)}{(0:a^{n-1})}=E(I;M)_n.   
\end{eqnarray*}
Now, suppose that $A=\Bbbk[\![X]\!]/(X^p)$ and let $a=\overline{X}$ be the class of $X$ in $A$. Then, 
$(0:a^n)=A$, for all $n\geq p$. Thus, $E(I)_n=0$, for all $n\geq p+1$. Moreover,
for all $2\leq n\leq p$, $(0:a^n)=(a^{p-n})$ and $E(I)_n=(a^{p-n})/(a^{p-n+1})$, so $\rt(I;M)=p$. Furthermore, 
$\psi_n(E(I;M)_{n+1})=(a(a^{p-n-1})+(a^{p-n+1}))/(a^{p-n+1})=E(I;M)_n$. Thus,
$\psi_n$ is surjective and $E(\agrw{I}{M})_n=0$, for all $2\leq n\leq p-1$. Thus, $\st(\agrw{I}{M})=2$, whereas
$\st(I;P)=\rt(I;M)=p$.
\end{example}

\begin{example}\label{classical} (See \cite[Example~5.1]{mp})
Let $(A,\mathfrak{m})$ be a noetherian local ring. Let $a_1,a_2$ a regular
$A$-sequence and $p\geq 2$. Set $x_{1}=a_{1}^{p}$, $x_{2}=a_{2}^{p}$
and $y=a_{1}a_{2}^{p-1}$. Let $I$ be the ideal generated by
$x_{1},x_{2},y$. Then, 
\begin{eqnarray*}
\SD(I)=\{1,2,\ldots,p\}\phantom{+}\mbox{ and }\phantom{+}\st(I)=\rt(I)=p.
\end{eqnarray*}
Indeed, set $V=A[X_{1},X_{2},Y]$ and let $\varphi:V\to
\rees{I}$ be the presentation of $\rees{I}$ sending $X_{i}$ to $x_it$
and $Y$ to $yt$. A minimal generating set of the ideal
$\ker(\varphi)$ is obtained from:
\begin{itemize}
\item a unique equation
  $F_n(X_1,X_2,Y)=a_1^{p-n}Y^n-a_2^{p-n}X_1X_2^{n-1}$ of
  degree $n$, for each $n$, $2\leq n\leq p$;
\item two equations $F_1(X_1,X_2;Y)=a_1^{p-1}Y-a_2^{p-1}X_1$ and
  $G_1(X_1,X_2,Y)=a_2Y-a_1X_2$ of degree 1.
\end{itemize}
\end{example}

Perturbing the ideal $I$ in Example~\ref{classical} with an extra regular parameter $a_3$ and considering the ideal $J=(a_3)$ we get

\begin{example}\label{wang} (See \cite[Example~6.1]{wang1})
Let $(A,\mathfrak{m})$ be a 3-dimensional regular local ring, $\mathfrak{m}=(a_1,a_2,a_3)$. 
Let $p\geq 2$, $x_1=a_1^p$, $x_2=a_2^p$, $y=a_1a_2^{p-1}+a_3^p$. Let $I$ be the ideal generated by 
$x_1,x_2,y$ and let $J$ be the ideal generated by $a_3$. 
Then, 
\begin{eqnarray*}
\war(J,A;I)=2,&\mar(J,A;I)=p,&\sar(J,A;I)=p,\\
&\st(I;A/J)=p,&\rt(I;A/J)=p.
\end{eqnarray*}
Since $\rad(I)=\mathfrak{m}$, it follows that $I$ is an
$\mathfrak{m}$-primary ideal. Thus, $\grade(I)=\height(I)=3$. 
By \cite[Corollary~1.6.19]{bh}, $x_1,x_2,y$ 
is a regular $A$-sequence. Thus, $I$ is an ideal of linear type (see, e.g., \cite{hsv2}). 
By Proposition~\ref{prop-linear}, $E(I;A/J)_n=E(J,A;I)_n$, for all $n\geq 2$, so
$\mar(A,J;I)=\st(I;A/J)$ and $\sar(J,A;I)=\rt(I;A/J)$.

Example~\ref{classical} says that $E(I;A/J)_m$ is generated by the coset
$a_1^{p-m}Y^m-a_2^{p-m}X_1X_2^{m-1}$, for all $2\leq m\leq p$. Moreover, 
$E(I;A/J)_n=0$, for all $n\geq p+1$. By Proposition~\ref{prop-key}, 
we deduce that the class of $u_m:=a_1^{p-m}y^m-a_2^{p-m}x_1x_2^{m-1}$
spans $E(J,A;I)_m=I^m\cap J/I(I^{m-1}\cap J)$, 
for all $2\leq m\leq p$. Furthermore,
$E(J,A;I)_n=0$, for all $n\geq p+1$. Therefore, $\mar(A,J;I)=\st(I;A/J)=p$ and
$\sar(J,A;I)=\rt(I;A/J)=p$.

Let us prove that $\war(J,A;I)=2$. If we see that
$I^m\cap J\subseteq I^{m-2}J$, for all $2\leq m\leq p$, then,
using the equalities in \eqref{eq-delta}, it follows that, for all $n\geq p+1$, 
\begin{eqnarray*}
I^n\cap J=I^{n-p}(I^p\cap J)\subseteq I^{n-p}(I^{p-2}J)=I^{n-2}J.
\end{eqnarray*}
So, let us prove that $I^m\cap J\subseteq I^{m-2}J$, for all $2\leq m\leq p$.
For such an $m$, we have
\begin{eqnarray*}
I^m\cap J=u_mA+I(I^{m-1}\cap J).     
\end{eqnarray*}
In particular, if $2\leq m\leq p$,  
\begin{eqnarray*}
I^m\cap J=u_mA+I(I^{m-1}\cap J)&=&u_mA+u_{m-1}I+I^2(I^{m-2}\cap J)\\
&=&u_mA+u_{m-1}I+\cdots +u_{3}I^{m-3}+I^{m-2}(I^2\cap J).    
\end{eqnarray*}
Hence, it is enough to see that $u_m\in I^{m-2}J$, for all $2\leq m\leq p$. 
A computation with {\sc Singular} \cite{Sing} says that 
\begin{eqnarray*}
&&u_m=a_1^{p-m}y^m-a_2^{p-m}x_1x_2^{m-1}=\sum_{j=0}^{m-2}v_j(x_2^{m-2-j}y^j)a_3\in (x_2,y)^{m-2}J,\mbox{ where}\\
&&v_0=a_1^{p-1}a_2^{p-m+1}a_3^{p-1},\phantom{+}v_1=a_1^{p-2}a_2^{p-m+2}a_3^{p-1},\phantom{+}\ldots,
v_{m-3}=a_1^{p-m+2}a_2^{p-2}a_3^{p-1}\mbox{ and }\\
&&v_{m-2}=(2a_1^{p-m+1}a_2^{p-1}a_3^{p-1}+a_1^{p-m}a_3^{2p-1}).
\end{eqnarray*}
\end{example}

Consider in Example~\ref{classical} an extra parameter $a_3$ and the ideal $J=(a_3)$, which is 
``normally transversal to $I$'' (see \cite[Definition~5.5.1 and Proposition~5.5.7]{wolmer2}). We have

\begin{example}\label{s<rt} 
Let $A=\Bbbk[a_1,a_2,a_3]$ be the polynomial ring over a field $\Bbbk$. 
Let $p\geq 2$, $x_{1}=a_{1}^{p}$, $x_{2}=a_{2}^{p}$, $y=a_{1}a_{2}^{p-1}$. 
Let $I$ be the ideal generated by $x_{1},x_{2},y$ and let $J$ be the ideal generated by $a_3$. Then 
\begin{eqnarray*}
\war(J,A;I)=1,&\mar(J,A;I)=1,&\sar(J,A;I)=1,\\
&\st(I;A/J)=p,&\rt(I;A/J)=p.
\end{eqnarray*}
By Example~\ref{classical}, $\st(I;A/J)=\rt(I;A/J)=p$. Clearly, 
$I^n\cap J=I^nJ=I(I^{n-1}\cap J)$, for all $n\geq 2$, 
so $E(J,A;I)_n=0$ and $\war(J,A;I)=\mar(J,A;I)=\sar(J,A;I)=1$.
\end{example}

\begin{example}\label{classical-variant} (See \cite[Example~5.2]{mp})
Let $(A,\mathfrak{m})$ be a noetherian local ring. Let $a_1,a_2$ a regular
$A$-sequence and let $p\geq 5$ be an odd integer. Set
$x_{1}=a_{1}^{p}$, $x_{2}=a_{2}^{p}$ and $y=a_{1}^2a_{2}^{p-2}$. Let
$I$ be the ideal generated by $x_{1},x_{2},y$. Then, 
\begin{eqnarray*}
\SD(I)=\{1,2,\ldots,(p-1)/2,(p+1)/2,p\},\phantom{+}\st(I)=(p+3)/2\phantom{+}\mbox{ and }\phantom{+}\rt(I)=p. 
\end{eqnarray*}
Concretely, set $V=A[X_{1},X_{2},Y]$ and let $\varphi:V\to
\rees{I}$ be the presentation of $\rees{I}$ sending $X_{i}$ to $x_it$
and $Y$ to $yt$. A minimal generating set of equations of $\rees{I}$ is
\begin{itemize}
\item an equation $F_p(X_1,X_2,Y)=Y^p-X_1^2X_2^{p-2}$ of degree
  $p$;
\item an equation $F_n(X_1,X_2,Y)=a_2Y^n-a_1X_1X_2^{n-1}$ of degree
  $n=(p+1)/2$;
\item an equation
  $F_n(X_1,X_2,Y)=a_1^{p-2n}Y^n-a_2^{p-2n}X_1X_2^{n-1}$ of degree $n$,
  for each $n$, $2\leq n\leq (p-1)/2$;
\item two equations $F_1(X_1,X_2;Y)=a_1^{p-2}Y-a_2^{p-2}X_1$ and
  $G_1(X_1,X_2,Y)=a_2^2Y-a_1^2X_2$ of degree 1.
\end{itemize}
\end{example}

Perturbing the ideal $I$ in Example~\ref{classical-variant} with an extra regular parameter $a_3$ and considering the ideal $J=(a_3)$
we get

\begin{example}
Let $(A,\mathfrak{m})$ be a 3-dimensional regular local ring, $\mathfrak{m}=(a_1,a_2,a_3)$. 
Let $p\geq 5$ be an odd integer, $x_{1}=a_{1}^{p}$, $x_{2}=a_{2}^{p}$, $y=a_{1}^2a_{2}^{p-2}+a_3^p$. 
Let $I$ be the ideal generated by $x_{1},x_{2},y$ and let $J$ be the ideal generated by $a_3$. Then, 
\begin{eqnarray*}
\war(J,A;I)=2,&\mar(J,A;I)=(p+3)/2,&\sar(J,A;I)=p,\\
&\st(I;A/J)=(p+3)/2,&\rt(I;A/J)=p.
\end{eqnarray*}
The argument is similar to that of Example~\ref{wang}. Since $I$ is an ideal of linear type,  
$\mar(J,A;I)=\st(I;A/J)$ and $\sar(J,A;I)=\rt(I;A/J)$. Example~\ref{classical-variant} gives the value of
$\st(I;A/J)$ and $\rt(I;A/J)$. 
\end{example}

Now, consider in Example~\ref{classical-variant} an extra parameter $a_3$ and the ideal $J=(a_3)$, which is 
normally transversal to $I$. Then, we get

\begin{example}\label{s,st<rt} 
Let $A=\Bbbk[a_1,a_2,a_3]$ be the polynomial ring over a field $\Bbbk$. 
Let $p\geq 2$, $x_{1}=a_{1}^{p}$, $x_{2}=a_{2}^{p}$, $y=a_{1}^2a_{2}^{p-2}$. 
Let $I$ be the ideal generated by $x_{1},x_{2},y$ and let $J$ be the ideal generated by $a_3$. Then, 
\begin{eqnarray*}
\war(J,A;I)=1,&\mar(J,A;I)=1\phantom{+++++},&\sar(J,A;I)=1,\\
&\st(I;A/J)=(p+3)/2\phantom{+},&\rt(I;A/J)=p.
\end{eqnarray*}
This follows from Example~\ref{classical-variant} and the fact that for all $n\geq 1$, $I^n\cap J=I^nJ$.
\end{example}

\begin{remark}\label{remark-dist}
It would be interesting to find an example of a triple $(J,A;I)$ for which
all five invariants of Theorem~\ref{thm-summ} were distinct. 
\end{remark}
The following easy lemma says where (not) to start when looking for such an example. 
\begin{lemma}
Let $I$ and $J$ be two ideals of $A$. 
\begin{itemize}
\item[$(1)$] If $\war(J,A;I)=1$ and $J\subset I$, then $\mar(J,A;I)=\sar(J,A;I)=1$.
\item[$(2)$] If $I$ and $J$ are normally transversal, then 
\[\war(J,A;I)=\mar(J,A;I)=\sar(J,A;I)=1.\]
\item[$(3)$] If $\sar(J,A;I)=1$, then $\rt(I;A/J)\leq \rt(I)$ and $\st(I;A/J)\leq \st(I)$.
\end{itemize}
\end{lemma}
\begin{proof} 
If $\war(J,A;I)=1$, then $I^n\cap J\subseteq I^{n-1}J$, for all $n\geq 1$. 
Since $J\subset I$, it follows that $I^n\cap J=I^{n-1}J$, for all $n\geq 2$. Therefore, for all $n\geq 2$, 
\begin{eqnarray*}
I^n\cap J=I^{n-1}J=I(I^{n-2}J)=I(I^{n-1}\cap J).
\end{eqnarray*}
Thus, $E(J,A;I)_n=0$, for all $n\geq 2$ and $\mar(J,A;I)=\sar(J,A;I)=1$. This proves $(1)$. 
If $I$ and $J$ are normally transversal, then, for all $n\geq 2$, $I^n\cap J=I^nJ=I(I^{n-1}\cap J)$,  
so $E(J,A;I)_n=0$ and $\war(J,A;I)=\mar(J,A;I)=\sar(J,A;I)=1$, which proves $(2)$. Finally, $(3)$ follows from
the rightmost inequality in \eqref{arrt} and Propositions~\ref{prop-arnumbers} and \ref{stmenor}.
\end{proof}

\section{A Fibonacci-like set of sifted degrees}\label{sec-fibo}

In the polynomial ring $A=\Bbbk[a_1,a_2]$ over a field $\Bbbk$, consider an ideal $I$ generated by 3 monomials 
of the same degree with no common factor. One can assume, without loss of generality, that $I$ is generated by $x_{1}=a_{1}^{p}$, $x_{2}=a_{2}^{p}$ and $y=a_{1}^ua_{2}^{p-u}$ for two positive integers $p$ and $u$ such that 
$\gcd(u,p)=1$ and 
$u<p/2$.

\begin{proposition}\label{3monomials} 
Let $A=\Bbbk[a_1,a_2]$ be the polynomial ring over a field $\Bbbk$. Let $p,u\in\mathbb{N}$, 
$p\geq 3$, $1\leq u<p/2$, and $\gcd(u,p)=1$. 
Let $x_{1}=a_{1}^{p}$, $x_{2}=a_{2}^{p}$, $y=a_{1}^ua_{2}^{p-u}$. 
Let $I$ be the ideal generated by $x_{1}=a_{1}^{p}$, $x_{2}=a_{2}^{p}$ and $y=a_{1}^ua_{2}^{p-u}$.
Let $q_1,\ldots,q_{n-1}$ be the quotients in Euclid's algorithm applied to $r_0:=p$ and $r_1:=u$, i.e., for every 
$1\leq i\leq n-1$, $r_{i-1}=q_ir_i+r_{i+1}$, with $q_i,r_{i+1}\in \mathbb{N}$, and $r_0>r_1\ldots>r_{n-1}>r_n=0$. 
Then, 
\begin{eqnarray}\label{sdi}
\SD(I)=\{d_{1,1},\ldots,d_{1,q_1}\}\cup\{d_{2,1},\ldots,d_{2,q_2}\}\cup\ldots\cup\{d_{n-1,1},\ldots,d_{n-1,q_{n-1}}\},
\end{eqnarray}
where each subset $\{d_{i,1},\ldots,d_{i,q_i}\}$ is a finite arithmetic progression with:
\begin{itemize}
\item $d_{1,j}-d_{1,j-1}=1$ and $d_{1,1}=1$;
\item $d_{2,j}-d_{2,j-1}=d_{1,q_1}$ and $d_{2,1}=d_{1,q_1}+1$;
\item $d_{i,j}-d_{i,j-1}=d_{i-1,q_{i-1}}$ and $d_{i,1}=d_{i-1,q_{i-1}}+d_{i-2,q_{i-2}}$, for every 
$3\leq i\leq n-1$.
\end{itemize}
Moreover, $d_{n-1,q_{n-1}}=p$, so $\st(I)=q_1+\cdots+q_{n-1}$ and $\rt(I)=p$.
\end{proposition}
\begin{proof}
Observe that $n\geq 2$ and, since $u<p/2$, $q_1\geq 2$. Furthermore, $q_i\geq 1$, for every $2\leq i\leq n-1$. 
A minimal generating set of the defining ideal of the Rees algebra of $I$ is given in \cite[Theorem~2.2]{cda}. 
In order to describe the degrees of these generators, i.e., the set $\SD(I)$ of sifted degrees of $I$, one has to focus on the sequence $(\vert\sigma_i-\tau_i\vert)$ introduced in \cite{cda} and, in particular, use 
\cite[Proposition 5.6]{cda} to check that for each sifted degree $\delta\geq 2$ of $I$ there is exactly one equation of degree $\delta$. A careful reading leads directly to the equality \eqref{sdi}.

Now, let us prove that $d_{n-1,q_{n-1}}=p$. 
For all $1\leq i\leq n-1$, let
$\delta_{i+1}$ be the last element in the arithmetic sequence 
$\{d_{i,1},\ldots,d_{i,q_i}\}$, i.e., $\delta_{i+1}=d_{i,q_{i}}$.
Note that,
for $2\leq i\leq n-1$, $d_{i,q_{i}}=d_{i,1}+(q_i-1)d_{i-1,q_{i-1}}$. On the other hand, 
$d_{i,1}=d_{i-1,q_{i-1}}+d_{i-2,q_{i-2}}$ if $i\geq 3$, and $d_{2,1}=d_{1,q_{1}+}1$.
Thus, for all $2\leq i\leq n-1$, one has that
$d_{i,q_i}=d_{i-2,q_{i-2}}+q_id_{i-1,q_{i-1}}$ if one sets $q_0=0$ and $d_{0,0}=1$, i.e., $\delta_{i+1}=q_i\delta_i+\delta_{i-1}$. By setting $\delta_0=0$ and $\delta_1=1$, one has that this inductive relation also holds for $i=1$, and the sequence $\{\delta_0,\ldots,\delta_n\}$
can also be defined recursively by 
\[\delta_0=0,\ \delta_1=1, \hbox{ and } \delta_{i+1}=q_i\delta_i+\delta_{i-1}, \hbox{ for all } 1\leq i\leq n-1\,.\]
Similarly, one can define a sequence $\{\mu_0,\ldots,\mu_n\}$ by setting
$\mu_0=1$, $\mu_1=0$, and considering the same inductive relation $\mu_{i+1}=q_i\mu_i+\mu_{i-1}$, for all $1\leq i\leq n-1$.
Note that the sequence $\{\delta_0,\ldots,\delta_n\}$ is completely determined by the 
quotients $\{q_1,\ldots,q_{n-1}\}$ of the Euclid's algorithm applied to $r_0=p$ and $r_1=u$. Moreover, since $\mu_2=\mu_0=1$, the sequence $\{\mu_1,\ldots,\mu_n\}$ is exactly the $\delta$-sequence  associated to the shorter list of quotients $\{q_2,\ldots,q_{n-1}\}$, i.e., the sequence of quotients obtained from Euclid's algorithm applied to $r'_0=u$ and $r'_1=r_2$ that has, hence, one division less.

On the other hand, one can easily show by induction that, for all $0\leq i\leq n$,
\[\delta_{i}u-\mu_i p=(-1)^{i-1}r_i\,.\]
For $i=n$, this gives the relation $\delta_n u=\mu_n p$.

Now, let us prove $\delta_n=p$ by induction on $n$. If $n=2$, then, necessarily, $u=1$, which corresponds to Example~\ref{classical}, so $\SD(I)=\{1,2,\ldots,p\}$ and, in particular, the largest integer in $\SD(I)$ is $p$. 
Applying the inductive hypothesis to the sequence $\{\mu_1,\ldots,\mu_n\}$, which corresponds to the Euclid's algorithm with $r'_0=u$ and $r'_1=r_2$, and which has one division less, we obtain $\mu_n=u$. Using the equality $\delta_n u=\mu_n p$, we deduce $\delta_n=p$.

From the description of $\SD(I)$ in \eqref{sdi} and the equality $d_{n-1,q_{n-1}}=p$, we deduce that $\st(I)=q_1+\cdots+q_{n-1}$ and $\rt(I)=p$.
\end{proof}

One can now play with possible quotients $q_1,\ldots,q_{n-1}$ in Euclid's algorithm to obtain several interesting families.

\begin{remark}\label{remu1}
Suppose that $u=1$ in Proposition~\ref{fibonacci}, which coincides with Example~\ref{classical}. Clearly, $u=1$ is equivalent to 
$n=2$. In this case, Euclid's algorithm has only one division, $r_0=q_1r_1+0$, with $r_0=p$, $r_1=u=1$ and $q_1=p$. Hence, $\SD(I)=\{1,2,\ldots,p\}$ and $\st(I)=\rt(I)=p$, as said before in Example~\ref{classical}. 
\end{remark}

\begin{remark}\label{remu2}
If $u\geq 2$ in Proposition~\ref{fibonacci}, then Euclid's algorithm applied to $r_0=p$ and $r_1=u$ will have at least two divisions (i.e., $n\geq 3$).
Suppose that $u=2$. Assume that $p$ is odd and $p\geq 5$, 
so that $1\leq u<p/2$ and $\gcd(u,p)=1$. Euclid's algorithm applied to $r_0=p$ and $r_1=2$ 
gives $r_0=q_1r_1+r_2$, for $q_1=(p-1)/2$ and $r_2=1$; and $r_1=q_2r_2$, for $q_2=2$. 
Thus, $n=3$, there are two divisions and $\SD(I)$ is the union of two finite arithmetic progressions. Note that $d_{1,q_1}=q_1=(p-1)/2$, 
$d_{2,1}=d_{1,q_1}+1=q_1+1=(p+1)/2$ and $d_{2,2}=d_{2,1}+d_{1,q_1}=p$. Therefore,
\begin{eqnarray*}
\SD(I)=\{1,2,\ldots,(p-1)/2\}\cup\{(p+1)/2,p\},\phantom{+}\st(I)=(p+3)/2\phantom{+}\mbox{and}\phantom{+}\rt(I)=p,
\end{eqnarray*}
as already observed in Example \ref{classical-variant}.
\end{remark}

\begin{example}\label{exup1}
In Proposition~\ref{fibonacci} take $p\geq 7$ odd and $u=(p-1)/2$. 
Note that $\gcd(u,p)=1$ (otherwise, a common factor to $u$ and $p$ would also be a factor of $p-u=(p+1)/2$ and this is impossible since $(p-1)/2$ and $(p+1)/2$ are consecutive integers). 
Thus, $1\leq u<p/2$ and $\gcd(u,p)=1$.
Euclid's algorithm applied to $r_0=p$ and $r_1=(p-1)/2$ gives $r_0=q_1r_1+r_2$, for $q_1=2$ and $r_2=1$; and $r_1=q_2r_2$, for $q_2=(p-1)/2$. 
So, $n=3$, there are two divisions and $\SD(I)$ is the union of two finite arithmetic progressions. 
Concretely, 
\begin{eqnarray*}
\SD(I)=\{1,2\}\cup\{2i+1\mid 1\leq i\leq (p-1)/2\},\phantom{+}\st(I)=(p+3)/2\phantom{+}\mbox{and}\phantom{+}\rt(I)=p.    
\end{eqnarray*}
Observe that while $\st(I)$ and $\rt(I)$ coincide with those in Remark~\ref{remu2}, $\SD(I)$ does not. 
\end{example}

\begin{example}\label{exup2}
In Proposition~\ref{fibonacci} take $p\geq 8$, $p$ a multiple of 4, and $u=(p-2)/2$. As in Example \ref{exup1}, if $p$ and $u$ had a common factor,
it would also be a factor of $p-u=(p+2)/2=u+2$, and hence it should be 2, which is impossible since $u$ is odd. Thus, $1\leq u<p/2$ and $\gcd(u,p)=1$.
Euclid's algorithm applied to $r_0=p$ and $r_1=(p-2)/2$ gives $r_0=q_1r_1+r_2$, for $q_1=2$ and $r_2=2$; $r_1=q_2r_2+r_3$, for $q_2=(p-4)/4$ and $r_3=1$; and $r_2=q_3r_3$, for $q_3=2$. So, $n=4$, there are three divisions and $\SD(I)$ is the union of three finite arithmetic 
progressions. In fact, 
\begin{eqnarray*}
\SD(I)=\{1,2\}\cup\{2i+1\mid 1\leq i\leq p/4\}\cup\{p\},\phantom{+}\st(I)=(p/4)+3\phantom{+}\mbox{and}\phantom{+}\rt(I)=p.      
\end{eqnarray*}
Note that $\rt(I)=4(\st(I)-3)$. 
\end{example}

\begin{example}
In Proposition~\ref{fibonacci} take $p=8(2\ell+1)$, for some $\ell\geq 1$, and $u=(p-4)/4$.
Again, if $p$ and $u$ had a common factor, it would also be a factor of $4u$, and hence it would divide $p-4u=4$, and this is impossible because $u$ is odd.
Thus, $1\leq u<p/2$ and $\gcd(u,p)=1$.
Since $u=4\ell+1$, Euclid's algorithm applied to $r_0=p$ and $r_1=u$ gives $r_0=q_1r_1+r_2$, for $q_1=4$ and 
$r_2=4$; $r_1=q_2r_2+r_3$, for $q_2=\ell$ and $r_3=1$; and $r_2=q_3r_3$, for $q_3=4$. 
Thus, $n=4$, there are three divisions and $\SD(I)$ is the union of three finite arithmetic progressions. One sees that:
\begin{eqnarray*}
\SD(I)=\{1,2,3,4\}\cup\{4i+1\mid 1\leq i\leq \ell\}\cup\{4\ell+5, 8\ell+6, 12\ell+7, 16\ell+8\},      
\end{eqnarray*}
and hence $\st(I)=\ell+8$ and $\rt(I)=16\ell+8=p$. Observe that $\rt(I)=8(2\st(I)-15)$.
\end{example}

\begin{example}\label{fibonacci}
Let $(F_i)_{i\geq 0}$ be the Fibonacci sequence defined by $F_0=0$, $F_1=1$ and $F_{i+2}=F_{i+1}+F_{i}$, for all $i\geq 0$.
Take $p=F_{\ell+2}$ and $u=F_{\ell}$, for some $\ell\geq 4$, in Proposition~\ref{fibonacci}. (The cases $\ell=1,2,3$ have already been considered in the previous families.) Let $I$ be generated by $x_{1}=a_{1}^{F_{\ell+2}}$, $x_{2}=a_{2}^{F_{\ell+2}}$ and $y=a_{1}^{F_{\ell}}a_{2}^{F_{\ell+1}}$. 
As $p-u=F_{\ell+1}$ and two consecutive terms in the Fibonacci sequence are relatively prime, one has that $1\leq u<p/2$ and $\gcd(u,p)=1$.
Applying Euclid's algorithm to $r_0=F_{\ell+2}$ 
and $r_1=F_{\ell}$ (recall that $u<p/2$), one has that $r_0=q_1r_1+r_2$, for $q_1=2$ and $r_2=F_{\ell-1}$;
$r_{i-1}=q_ir_i+r_{i+1}$, for $2\leq i\leq \ell-2$, with $q_i=1$ and $r_{i+1}=F_{\ell-i}$; 
and $r_{\ell-2}=F_{3}=2=q_{\ell-1}r_{\ell-1}$ with $q_{\ell-1}=2$ and $r_{\ell-1}=F_{2}=1$. 
Thus, $n=\ell$, there are $\ell-1$ divisions and $\SD(I)$ is the union of $\ell-3$ singletons and two subsets of
cardinality 2. One can check that 
\begin{eqnarray*}
\SD(I)=\{F_{2},F_3\}\cup\{F_4\}\cup\ldots\cup\{F_{\ell}\}\cup\{F_{\ell+1},F_{\ell+2}\},     
\end{eqnarray*}
and hence $\st(I)=\ell+1$ and $\rt(I)=F_{\ell+2}=p$.
Note that, for $\varphi=\frac{1+\sqrt{5}}{2}$, one has:
\[\rt(I)=F_{\st(I)+1}=\frac{\varphi^{\st(I)+1}-(-\varphi^{-1})^{\st(I)+1}}{\sqrt{5}}.\]
\end{example}

\begin{question}
We would like to characterize the finite subsets $D\subseteq\mathbb{N}$ 
for which there exists an ideal $I$ whose set of sifted degrees is $\SD(I)=D$. In addition, given such a subset $D$, we would like to find 
an ideal $I$ with $\SD(I)=D$. 
For instance, given the Fibonacci-like set $D=\{1,2,3,5,8,13,21\}$, we have seen that $\SD(x^{21},y^{21},x^8y^{13})=D$ 
(see Example~\ref{fibonacci}).
\end{question}

\section{More examples}\label{sec-moreEx}

We finish the present note with three more examples, now of a more theoretical nature. 
The first example is a strong and difficult result due to Simis, Ulrich and Vasconcelos, here restated using our ``sifted type'' terminology.
To our knowledge, this theorem is one of the first results in the literature intended to provide a family of ideals with a large relation type, but with a few distinct degree equations. 
In that sense, it could be thought as a seed to the definition of sifted type.

\begin{example} (See \cite[Theorem~4.10]{suv95})
Let $A$ be a Gorenstein local ring
with infinite residue field and let $I$ be a strongly Cohen-Macaulay ideal of $A$ of height $h\geq 2$ 
with minimal number of generators $n$, analytic spread $\ell$ and reduction number $r$.
Assume that $\ell<n$ and that $I$ satisfies $\mathscr{F}_1$ locally in codimension $n-2$ (see \cite{hsv2} or \cite{suv95}).
If $r\leq n-h$, then $\st(I)=2$ and $\rt(I)=n-h+1$. 
\end{example}

The next example is again due to Simis, Ulrich and Vasconcelos. It serves as
a counterexample to a Conjecture of Valla (see \cite{hmv}), by exhibiting a prime ideal 
generated by analytically independent elements, but not of linear type, nor syzygetic. 

\begin{example}{\rm (See \cite[Proposition~4.5]{suv93})
Let $\mathfrak{p}$ be the homogeneous prime ideal of $A=k[a_1,\ldots,a_9]$
generated by the polynomials:
\begin{eqnarray*}
&
a_5a_7+a_6a_9,\ \ a_3a_6a_8+a_1a_4a_9,\ \
a_1a_4a_5-a_3a_5a_8+a_2a_6a_8-a_1a_8a_9,&\\
&a_3a_5a_6+a_2a_5a_7+a_3a_5a_9+a_1a_9^2,\ \
a_3a_6^2+a_2a_6a_7+a_3a_6a_9-a_1a_7a_9,&\\
&a_3a_4a_6+a_2a_4a_7+a_3a_7a_8+a_3a_4a_9.&
\end{eqnarray*}
Then, $\st(\mathfrak{p})=\rt(\mathfrak{p})=2$. There are eight
linear equations for the Rees algebra, coming from the syzygy module of $\mathfrak{p}$, and one extra generador:}
\begin{eqnarray*}
a_1a_3X_1X_2+a_3X_2X_4-a_2X_2X_5+a_3X_3X_5+a_1a_2X_1X_6-a_1X_4X_6.
\end{eqnarray*}
\end{example}

Finally, this last example, due to Hartshorne, gives an irreducible surface in $\mathbb{C}^4$ which is not a (set theoretically a) complete intersection. 

\begin{example} (See \cite[Example~3.4.2]{hartshorne} and 
\cite[Example~3, p.~247]{geyer}) Let $V$ be the irreducible algebraic affine set $V=\{(s,st,t(t-1),t^2(t-1))\in\mathbb{A}^4\mid s,t\in k\}$, which is not a set-theoretically complete intersection. 
A run with {\sc Singular} says that $V=V(I)$, where $I$ is the ideal defined by the elements:
\[
f_1=x_1x_4-x_2x_3,\; f_2=x_1^2x_3+x_1x_2-x_2^2,\;
f_3=x_3^3+x_3x_4-x_4^2.
\]
Then, $I=(f_1,f_2,f_3)$ is an ideal of linear type of $k[x_1,x_2,x_3,x_4]$. 
Moreover, $I(V)=\mathfrak{p}$, where $\mathfrak{p}$ is the ideal defined by the polynomials $f_1,f_2,f_3$ and $f_4=x_1x_3^2+x_1x_4-x_2x_4$. Then, $\mathfrak{p}$ is a height 2 prime ideal of $k[x_1,x_2,x_3,x_4]$ with $\st(\mathfrak{p})=\rt(\mathfrak{p})=2$. Observe that these calculations contradict the equality $I(V)=I$ stated in Geyer \cite[p. 248]{geyer}.
\end{example}


\begin{thebibliography}{99}

\bibitem{agh}{I.M. Aberbach, L. Ghezzi, H.T. H\`a, 
Homology multipliers and the relation type of parameter ideals,
Pacific J. Math. {\bf 226} (2006), 1-39.}

\bibitem{br}{J. Barja, A.G. Rodicio,
Syzygetic ideals, regular sequences, and a question of Simis.
J. Algebra {\bf 121} (1989), 310–314.}

\bibitem{bg}{I. Bermejo, P. Gimenez, 
On Castelnuovo-Mumford regularity of projective curves. 
Proc. Amer. Math. Soc. {\bf 128} (2000), 1293–1299.}

\bibitem{bh}{W. Bruns, J. Herzog,
Cohen-Macaulay rings. 
Cambridge Studies in Advanced Mathematics, 39. Cambridge University Press, Cambridge, 1993
}

\bibitem{cda}{T. Cortadellas Benítez, C. D'Andrea,
The Rees algebra of a monomial plane parametrization.
J. Symb. Comput. {\bf 70} (2015), 71-105.}

\bibitem{costa1}{D.L. Costa, 
On the torsion-freeness of the symmetric powers of an ideal, 
J. Algebra {\bf 80} (1983), 152-158.}

\bibitem{costa2}{D.L. Costa. 
Sequences of linear type. 
J. Algebra {\bf 94} (1985), 256-263.}

\bibitem{Sing}{
 W. Decker, G.-M. Greuel, G. Pfister, H. Sch{\"o}nemann, 
{\sc Singular} {4-4-0} --- {A} computer algebra system for polynomial computations.
Available at {\tt http://www.singular.uni-kl.de} (2024).}

\bibitem{do}{A.J. Duncan, L. O'Carroll, 
A full uniform Artin-Rees theorem, 
J. reine angew. Math. {\bf 394} (1989), 203-207.}

\bibitem{eh}{D. Eisenbud, M. Hochster, 
A Nullstellensatz with nilpotents and Zariski's main lemma on holomorphic functions,
J. Algebra {\bf 58} (1979), 157-161.}

\bibitem{geyer}{W.D. Geyer, 
On the number of equations which are necessary to describre an algebraic set in $n$-space. 
Atas da 3a Escola de \'Algebra, Brasilia (1976), 183-317.}

\bibitem{gsvv}{P. Gimenez, A. Simis, W.V. Vasconcelos, R.H. Villarreal, 
On complete monomial ideals,
J. Commut. Algebra {\bf 8} (2016), 207–226.}

\bibitem{gp}{J.M. Giral, F. Planas-Vilanova,
Integral degree of a ring, reduction numbers and uniform Artin-Rees numbers,
J. Algebra {\bf 319} (2008), 3398–3418.}

\bibitem{hartshorne}{R. Hartshorne,
Complete intersections and connectedness.
Amer. J. Math. {\bf 84} (1962), 497-508.}

\bibitem{hmv}{M. Herrmann, B. Moonen, O. Villamayor, Ideals of linear type and some variants. The Curves Seminar at Queen's, Vol. VI, Exp. No. H, 37 pp., Queen's Papers in Pure and Appl. Math., {\bf 83}, Queen's Univ., Kingston, ON, 1989.}

\bibitem{hrz}{M. Herrmann, J. Ribbe, S. Zarzuela, On the Gorenstein property of Rees and form rings of powers of ideals. Trans. Amer. Math. Soc. {\bf 342} (1994), 631-643.}

\bibitem{herzog}{J. Herzog,
Homological properties of the module of differentials,
Cole\c{c}ao Atas Sociedade Brasileira de Mat. {\bf 14} (1981), 33-64.}

\bibitem{hsv1}{J. Herzog, A. Simis, W.V. Vasconcelos,
Approximation complexes of blowing-up rings.
J. Algebra {\bf 74} (1982), 466–493.}

\bibitem{hsv2}{J. Herzog, A. Simis, W.V. Vasconcelos,
Koszul homology and blowing-up rings. Commutative algebra (Trento, 1981), pp. 79–169, 
Heidelberger Taschenbücher, Lect. Notes Pure Appl. Math., 84, Dekker, New York, 1983.}

\bibitem{hsv3}{J. Herzog, A. Simis, W.V. Vasconcelos,
Approximation complexes of blowing-up rings. II.
J. Algebra {\bf 82} (1983), 53–83.}

\bibitem{huneke1}{C. Huneke, 
On the symmetric and Rees algebra of an ideal generated by a $d$-sequence, 
J. Algebra {\bf 62} (1980), 268-275.}

\bibitem{huneke2}{C. Huneke, 
The theory of $d$-sequences and powers of ideals, 
Adv. in Math. {\bf 46} (1982), 249--279.}

\bibitem{huneke3}{C. Huneke, 
Uniform bounds in noetherian rings, 
Invent. Math. {\bf 107} (1992), 203-223.}

\bibitem{huneke4}{C. Huneke, 
Desingularizations and the uniform Artin-Rees theorem, 
J. London Math. Soc. (2) {\bf 62} (2000), 740-756.}

\bibitem{ioss}{R. Iglesias, M. Orth, E. Sáenz de Cabezón, W.M. Seiler,
On the minimal free resolution of the Rees algebra of some monomial ideals, 
AAECC (2025), https://doi.org/10.1007/s00200-025-00680-y.
}

\bibitem{jk}{B.L. Johnson, D. Katz, 
On the relation type of large powers of an ideal, 
Mathematika {\bf 41} (1994), 209-214.}

\bibitem{jm}{M.R. Johnson, J. McLoud-Mann,
On equations defining Veronese rings, 
 Arch. Math. {\bf 86} (2006), 205-210.}

\bibitem{kuhl}{M. K\"uhl, 
On the symmetric algebra of an ideal. 
Manuscripta Math. {\bf 37} (1982), 49–60.}

\bibitem{lai}{Y.H. Lai, 
On the relation type of systems of parameters,
J. Algebra {\bf 175} (1995), 339-358.}

\bibitem{mcp}{J. McCullough, I. Peeva, 
Counterexamples to the Eisenbud-Goto regularity conjecture,
J. Amer. Math. Soc. {\bf 31} (2018), 473–496.}

\bibitem{micali1}{A. Micali, 
Alg\`ebre sym\'etrique d'un id\'eal. (French)
C. R. Acad. Sci. Paris 251 (1960), 1954–1956.} 

\bibitem{micali2}{A. Micali, 
Sur les alg\`ebres universelles. 
Annales Inst. Fourier 14 (1964), 33–88.}

\bibitem{mr}{A. Micali, N. Roby, 
Algèbres symétriques et syzygies. (French) J. Algebra {\bf 17} (1971), 460–469.}

\bibitem{mu}{S. Morey, B. Ulrich, 
Rees algebras of ideals with low codimension.
Proc. Amer. Math. Soc. {\bf 124} (1996), 3653–3651.}

\bibitem{mp}{F. Mui\~{n}os, F. Planas-Vilanova,
The equations of Rees algebras of equimultiple ideals of deviation one,
Proc. Amer. Math. Soc. {\bf 141} (2013), 1241-1254.}

\bibitem{ocarroll1}{L. O'Carroll, 
A uniform Artin-Rees theorem and Zariski's main lemma on holomorphic functions, 
Invent. Math. {\bf 90} (1987), 674-682.}

\bibitem{ocarroll2}{L. O'Carroll, 
A note on Artin-Rees numbers,
Bull. London Math. Soc. {\bf 23} (1991), 209-212.}

\bibitem{ocarroll3}{L. O'Carroll, 
Addendum to: A note on Artin-Rees numbers, 
Bull. London Math. Soc. {\bf 23} (1991), 555-556.}

\bibitem{op}{L. O'Carroll, F. Planas-Vilanova,
Irreducible affine space curves and the uniform Artin-Rees property on the prime spectrum. 
J. Algebra {\bf 320} (2008), 3339–3344.}

\bibitem{fp-camb}{F. Planas-Vilanova, 
On the module of effective relations of a standard algebra, 
Math. Proc. Camb. Phil. Soc. {\bf 124} (1998), 215-229.}

\bibitem{fp-crelle}{F. Planas-Vilanova,
The strong uniform Artin-Rees property in codimension one,
J. reine angew. Math. {\bf 527} (2000), 185-201.
}

\bibitem{fp-lisboa}{F. Planas-Vilanova,
An approach to the uniform Artin-Rees theorems from the notion of relation type. 
Commutative algebra, 175-191, Lect. Notes Pure Appl. Math., 244, Chapman \& Hall/CRC, Boca Raton, FL, 2006. 
}

\bibitem{fp-2022}{F. Planas-Vilanova,
Regular local rings of dimension four and Gorenstein syzygetic prime ideals. J. Algebra {\bf 601} (2022), 105-114.}

\bibitem{schenzel}{P. Schenzel, 
Castelnuovo's index of regularity and reduction numbers, 
Banach Center Publications {\bf 26} (2) (1990), 201-208.}

\bibitem{raghavan}{K. Raghavan, 
Powers of ideals generated by quadratic sequences,
Trans. Amer. Math. Soc. {\bf 343} (1994), 727–747.}

\bibitem{suv93}{A. Simis, B. Ulrich, V.V. Vasconcelos,
Jacobian dual fibrations.
Amer. J. Math. {\bf 115} (1993), 47–75. 
}

\bibitem{suv95}{A. Simis, B. Ulrich, V.V. Vasconcelos,
Cohen-Macaulay Rees algebras and degrees of polynomial relations.
Math. Ann. {\bf 301} (1995), 421-444. 
}

\bibitem{sv}{
A. Simis, W.V. Vasconcelos,
The syzygies of the conormal module.
Amer. J. Math. {\bf 103} (1981), 203–224.}

\bibitem{striuli}{J. Striuli, 
Strong Artin-Rees property in rings of dimension one and two,
Rend. Istit. Mat. Univ. Trieste {\bf 39} (2007), 325–335. }

\bibitem{trivedi}{V. Trivedi, 
Hilbert functions, Castelnuovo-Mumford regularity and uniform Artin-Rees numbers, 
Manuscripta Math. {\bf 94} (1997), 485-499.}

\bibitem{trung1}{N.V. Trung, 
Absolutely superficial elements, 
Math. Proc. Camb. Phil. Soc. {\bf 93} (1983), 35-47.}

\bibitem{trung2}{N.V. Trung, 
Reduction exponent and degree bound for the defining equations of graded rings, 
Proc. Amer. Math. Soc. {\bf 101} (1987), 229-236.}

\bibitem{trung3}{N.V. Trung, 
The Castelnuovo regularity of the Rees algebra and the associated graded ring, 
Trans. Amer. Math. Soc. {\bf 350} (1998), 2813-2832.}

\bibitem{valla}{G. Valla, 
On the symmetric and Rees algebras of an ideal, Manuscripta Math. {\bf 30} (1980), 239--255.}

\bibitem{wolmer1}{W.V. Vasconcelos,
On the equations of Rees algebras,
J. reine angew. Math. {\bf 418} (1991), 189-218.}

\bibitem{wolmer2}{W.V. Vasconcelos, 
{\em Arithmetic of Blowup Algebras}, Cambridge University Press, Cambridge, 1994.}

\bibitem{wolmer3}{W.V. Vasconcelos, {\em Computational Methods in
Commutative Algebra and Algebraic Geometry}, Springer-Verlag, Berlin,
1998.}

\bibitem{wolmer4}{W.V. Vasconcelos, 
Cohomological degrees of graded modules. Six lectures on commutative algebra (Bellaterra, 1996), 345-392,
Progr. Math., 166, Birkhäuser, Basel, 1998.}

\bibitem{wolmer5}{W.V. Vasconcelos, 
{\em Integral closure. Rees algebras, multiplicities, algorithms}. 
Springer Monographs in Mathematics. Springer-Verlag, Berlin, 2005.
}

\bibitem{wang1}{H.J. Wang, 
Some uniform properties of 2-dimensional local rings, 
J. Algebra {\bf 188} (1997), 1-15.}

\bibitem{wang2}{H.J. Wang, 
The Relation-Type Conjecture holds for rings with finite local cohomology, 
Comm. in Algebra {\bf 25} (1997), 785-801.}

\bibitem{santi}{S. Zarzuela, Complete intersection augmented algebras. Math. Ann. {\bf 306} (1996), 159-168. }

\end{thebibliography}
\end{document}